%% file: main.tex
\pdfoutput=1
\documentclass[reqno,11pt]{amsart}

\usepackage[letterpaper, margin=1in]{geometry}
\usepackage{microtype}

\usepackage{amsmath, amssymb, amsthm, bm}
\usepackage[shortlabels]{enumitem}

\usepackage{hyperref}

\usepackage{parskip}
\usepackage[center]{caption}

\usepackage[utf8]{inputenc}
\usepackage{float}
\newtheorem{theorem}{Theorem}[section]
\newtheorem{corollary}[theorem]{Corollary}

\newtheorem{lemma}[theorem]{Lemma}
\newtheorem{prop}[theorem]{Proposition}

\theoremstyle{definition}
\newtheorem{definition}{Definition}[section]
\usepackage[noabbrev,capitalize]{cleveref}
	\crefname{claim}{Claim}{Claims}
	\Crefname{claim}{Claim}{Claims}
	\crefname{app-corollary}{Corollary}{Corollaries}
	\Crefname{app-corollary}{Corollary}{Corollaries}
	\crefname{app-definition}{Definition}{Definitions}
	\Crefname{app-definition}{Definition}{Definitions}
	\crefname{figure}{Figure}{Figures}
	\Crefname{figure}{Figure}{Figures}
	\crefname{lemma}{Lemma}{Lemmata}
	\Crefname{lemma}{Lemma}{Lemmata}
	\crefname{app-lemma}{Lemma}{Lemmata}
	\Crefname{app-lemma}{Lemma}{Lemmata}
	\crefname{app-proposition}{Proposition}{Proposition}
	\Crefname{app-proposition}{Proposition}{Proposition}
	\crefname{app-theorem}{Theorem}{Theorems}
	\Crefname{app-theorem}{Theorem}{Theorems}

\newcommand{\rn}{\mathbb{R}^n}

\newcommand{\rd}{\mathbb{R}^d}
\newcommand{\one}{\mathbf{1}}
\newcommand{\onep}{\mathbf{1}^{\perp}}
\newcommand{\bg}{\overline{B_G}}
\DeclareMathOperator{\rank}{rank}

\makeatletter
\newcommand{\setword}[2]{%
  \phantomsection
  #1\def\@currentlabel{\unexpanded{#1}}\label{#2}%
}
\makeatother

\title{Spectral Conditions for Spherical 2-Distance Sets}

\author{Iliyas Noman}
\author{Yuan Yao}
\thanks{Noman was sponsored by HSSIMI, Sts. Cyril \& Methodius International Foundation, and Foundation for support of children Mariyana Palanchova. Yao was supported by the MIT Math Department.}

\email{iliyas@mit.edu, yyao1@mit.edu}

\begin{document}

\begin{abstract}
\input{abstract}
\end{abstract}

\maketitle

\input{paper}
\bibliographystyle{style}

\bibliography{biblio}

\end{document}

%% file: abstract.tex
A set of points $S$ in $d$-dimensional Euclidean space $\mathbb{R}^d$ is called a \emph{2-distance set} if the set of pairwise distances between the points has cardinality two. The 2-distance set is called \emph{spherical} if its points lie on the unit sphere in $\mathbb{R}^{d}$. We characterize the spherical 2-distance sets using the spectrum of the adjacency matrix of an associated graph and the spectrum of the projection of the adjacency matrix onto the orthogonal complement of the all-ones vector. We also determine the lowest dimensional space in which a given spherical 2-distance set could be represented using the graph spectrum.

%% file: paper.tex
\section{Introduction}
The problem of sets with few distances is considered to be naturally emerging in discrete geometry and coding theory.
The simplest case is the one of \emph{2-distance sets}, i.e. sets of points in $n$-dimensional space with only two possible distances between its elements.

The study of 2-distance sets began in 1947 when Kelly \cite{kelly1947elementary} showed that a 2-distance set in the plane can have at most 5 points. In 1963, Croft \cite{croft19639} showed that a 2-distance set in $\mathbb{R}^3$ can have at most 6 points. Larman, Rogers, and Seidel \cite{larman1977two} found an upper bound on the size of a 2-distance set in $\mathbb{R}^d$ that is asymptotically tight, essentially resolving this question. 

The other extreme case, the ``minimal'' 2-distance sets, has also been extensively studied.
There are infinitely many non-isomorphic 2-distance sets with $d+1$ points or fewer, so the smallest interesting case is $d+2$ points. 
In 1966, Einhorn and Schoenberg \cite{einhorn1966euclidean} showed that there is a finite number of 2-distance sets in $\rd$ with $d+2$ points and gave a combinatorial interpretation for the number of such ``small" sets (see \cref{main}). 


Motivated by \cref{main}, we are interested in characterizing the 2-distance sets whose points lie on the unit sphere. We call those sets \emph{spherical}. A previous characterization of the spherical 2-distance sets was completed in 2012 by Nozaki and Shinohara \cite{nozaki2012geometrical}, who gave a necessary and sufficient condition for a graph to correspond to a spherical 2-distance set in some Euclidean space (see Definition~\ref{sphericalrep}) based on Roy's previous results \cite{roy2010minimal}. The conditions in \cite[Theorem~2.4]{nozaki2012geometrical} involve bulky expressions that use the eigenvalues of the adjacency matrix and the angles of its eigenvectors with the all-ones vector.
In contrast, here we focus on the ``small" spherical graphs (with exactly $d+2$ vertices in $\rd$) and obtain a much cleaner necessary and sufficient condition for a graph to have a spherical representation in this case, only using the eigenvalues of the adjacency matrix and the eigenvalues of its projection onto a subspace.
\begin{theorem}\label{mainres}
    Let $d$ a be positive integer, and let $G$ be a graph on $d+2$ vertices whose adjacency matrix $A_G$ has eigenvalues $\lambda_1 \geq \lambda_2 \geq \dots \geq \lambda_{d+2}$. Let $J$ be the all-ones matrix and $P = I - \frac{1}{d+2}J$ be the projection matrix onto the subspace orthogonal to the all-ones vector. Then $G$ has a spherical representation in $\rd$ if and only if the maximum eigenvalue of $PA_GP$ is equal to $\lambda_2$ and the multiplicity of the eigenvalue $\lambda_2$ is equal in $A_G$ (excluding $\lambda_1$ if $\lambda_1 = \lambda_2$) and $PA_GP$.
\end{theorem}

In fact, the lowest dimensional space in which a given spherically representable graph $G$ is also determined by the multiplicity of the second eigenvalue of $G$.

\begin{prop}\label{lowestdimensionspherical}
    Let $n \leq d$ be a positive integer. If a graph $G$ on $d+2$ vertices is spherically representable in $\rd$, then it is (spherically) representable in $\rn$ if and only if the multiplicity of $\lambda_2$ in its adjacency matrix $A_G$ (excluding $\lambda_1$ if $\lambda_1 = \lambda_2$) is at least $d+1-n$. 
\end{prop}

The paper is organized as follows. In Section~\ref{preliminaries}, we define (spherical) 2-distance sets and (spherically) representable graphs, and recall two results from the literature that we use in the paper. In Section~\ref{conditions}, we give four initial conditions for a graph on $d+2$ vertices to be representable as a spherical 2-distance set in $\rd$. In Section~\ref{secondeigen}, we rephrase the four conditions from Section~\ref{conditions} in terms of another matrix, and find a formula for the matrix that depends on the second eigenvalue of the adjacency matrix of the graph. In Section~\ref{reduction}, we further reduce the positive-semidefinite conditions, which we use to prove Theorem~\ref{mainres} and Proposition~\ref{lowestdimensionspherical} in Section~\ref{finsec}.

\section{Preliminaries}\label{preliminaries}
First, we formally define a 2-distance set and some related concepts.
\begin{definition}
A set of points $S$ in $\rd$ is a \emph{2-distance set} if 
\begin{equation*}
    |D| = 2, \text{ where } D = \{ ||p_i - p_j|| \text{ for } p_i, p_j(\neq p_i) \in S\}.
\end{equation*}
Let $D = \{\alpha_1, \alpha_2\}$ such that $\alpha_1 > \alpha_2$. Denote by $k=\alpha_1 / \alpha_2 > 1$ the \emph{distance ratio} of $S$. Two 2-distance sets are \emph{equivalent} if one can obtain one from the other via isometries and scaling.
\end{definition}

The structure of a 2-distance set can be naturally encoded with a graph.

\begin{definition}
For a 2-distance set $S$, we define its \emph{associated graph} $G = G(S)$ to be a graph whose vertices correspond  points in $S$, and two vertices are adjacent if and only if the distance between their corresponding points in $S$ is $\alpha_1$ (i.e. the longer distance of the two).
\end{definition}

We are also interested in the opposite direction, i.e. when a graph has an associated 2-distance set.

\begin{definition}
A graph $G$ is \emph{representable in $\rd$ with ratio $k>1$} if there exists a 2-distance set $S$ in $\rd$ with distance ratio $k>1$ such that $G = G(S)$ is the associated graph of $S$.
\end{definition}

There are infinitely many non-equivalent 2-distance sets on at most $d+1$ points. For instance, we can generate infinitely many 2-distance sets by taking a regular simplex on $d$ points in a hyperplane of dimension $d-1$ and an arbitrary point on the line through the center of the simplex that is perpendicular to the $(d-1)$-dimensional hyperplane of the simplex. 

The following theorem gives a necessary and sufficient condition for a graph of size $d+2$ to be representable in $\rd$, thus proving that there are finitely many 2-distance sets with $d+2$ points. It is a direct implication from the proof of Theorem 1 in \cite{einhorn1966euclidean}.

\begin{theorem}[Einhorn--Schoenberg \cite{einhorn1966euclidean}]\label{main}
    Let $G$ be a graph on $d+2$ vertices. If $G$ is a complete multipartite graph, then $G$ is not representable in $\rd$ for any value of $k>1$. If $G$ is not a complete multipartite graph, then there exists a unique value $k>1$ for which $G$ is representable in $\rd$ with ratio $k$. 
\end{theorem}
\cref{main} allows us to work with the graph of the 2-distance set, which is a much simpler structure than the 2-distance set. We also cite the following lemma which was used to prove Theorem~\ref{main}. It gives us a necessary and sufficient condition for the existence of points with a fixed set of distances between them in a Euclidean space of dimension $d$.

\begin{lemma}[Schoenberg \cite{schoenberg1935remarks}] \label{ns}
Let $\{m_{ij}\}_{i,j = 1}^n$ be nonnegative real numbers such that $m_{ij} = m_{ji}$ for all $i,j$ and $m_{ii} = 0$ for $i = 1,2\dots n$. There exist points $p_1, \dots, p_n \in \mathbb{R}^d$ with $||p_i - p_j|| = m_{ij}$ for all $i, j$ if and only if the $(n-1) \times (n-1)$ matrix $B'$ with 
$$b_{ij}' = m_{1i}^2 + m_{1j}^2 - m_{ij}^2$$
is positive semidefinite (which we denote by $B' \succeq 0$) and $\rank(B') \leq d$. Moreover, this configuration is unique up to congruence.
\end{lemma}
Now we give a definition for a spherical 2-distance set, the main focus of this paper. A similar definition is used in the literature (\cite{musin2009spherical}, \cite{musin2019graphs}).

\begin{definition}\label{sphericalrep}
A 2-distance set $S$ in $\rd$ is called \emph{spherical} if all of its points lie on an $(d-1)$-dimensional sphere in $\rd$. A graph $G$ is \emph{spherical} or has a \emph{spherical representation} in $\rd$ if there exists a spherical 2-distance set $S\subset \rd$ whose associated graph is $G$.
\end{definition}

Finally, we denote by $A_G$ the adjacency matrix of a graph $G(V,E)$ and its eigenvalues by $\lambda_1 \geq \lambda_2 \geq \dots \geq \lambda_{|V|}$. For the remainder of the paper, $\lambda_i$ will always refer to the $i$-th eigenvalue of $A_G$.

\section{Conditions for spherical 2-distance sets}\label{conditions}

In this section we give a set of four necessary and sufficient conditions for a graph to have a spherical representation. Note that similar conditions have been studied for general spherical point sets, see \cite{tarazaga1996circum} and \cite{neumaier1981distance}.

First, we define the matrix $B_G$ which encodes the pairwise distances of the corresponding 2-distance set. 
\begin{definition}
Given a \emph{representable} graph $G(V, E)$ on $d+2$ vertices with ratio $k$, denote by $B_G = B_G(k)$ the $(d+2)\times(d+2)$ matrix with entries
    $$b_{ij} = 
    \begin{cases} 
        0 \text{, if } i = j\\
        -1 \text{, if } \{ i,j\} \not \in E\\
        -k^2 \text{, if } \{ i,j\} \in E.
    \end{cases}
$$
\end{definition}
Next, we gives an exact characterization of graphs with $d+2$ vertices that are representable in $\rd$. This characterization is symmetric with respect to all vertices, in contrast with Lemma~\ref{ns}, in which one vertex has to be used as a reference. Let $\one$ be the all-ones vector and $\one^{\perp}$ be the orthogonal complement of the span of $\one$. 
\begin{prop} \label{1and2}
    Let $G$ be a graph on $d+2$ vertices. Then $G$ is representable in $\rd$ with ratio $k$ if and only if the following two conditions hold.
    \begin{enumerate}[label={$(\arabic*)$}]
        \item For all vectors $w\in \one^{\perp}$ we have $w^TB_G(k)w \geq 0$.
        \item There exists a nonzero vector $w\in \one^{\perp}$ such that $B_G(k)w = \gamma \one$ for a real number $\gamma$.
    \end{enumerate} 

\end{prop}
\begin{proof}
We define $B_G' = \left(I - \one e_1^T\right)B_G\left(I - e_1\one^T\right)$, which is equivalent to the definition in Lemma~\ref{ns}. Thus, $G$ is representable in $\rd$ if and only if 
$$B_G' \succeq 0 \text{ and } \rank(B_G') \leq d.$$
We show that these two conditions are equivalent to conditions \ref{eq: 1} and \ref{eq: 2}, respectively.

The condition $B_G' \succeq 0$ is equivalent to $v^T B_G' v \geq 0$ for all vectors $v$. This gives us
$$v^T \left(I - \one e_1^T\right)B_G\left(I - e_1\one^T\right) v \geq 0$$
$$	\Leftrightarrow w^TB_Gw \geq 0, \text{ where } w = \left(I - e_1\one^T\right) v.$$
Note that $w_1 = -\sum_{i = 2}^n v_{i}, w_2 = v_2, w_3 = v_3, \dots, w_{d+2} = v_{d+2},$
so $w$ is a vector in $\one^{\perp}$ and hence condition \ref{eq: 1} implies $B_G' \succeq 0$. Conversely, every vector in $\onep$ can be expressed in this form, so $B_G' \succeq 0$ implies condition  \ref{eq: 1} as well.

The condition $\rank(B_G') \leq d $ is equivalent to $\dim \operatorname{null}(B_G') \geq 2$. Observe that $e_1 \in \operatorname{null}(B_G')$, so there exists a vector $v$ independent with $e_1$ such that $\left(I - \one e_1^T\right)B_G\left(I - e_1\one^T\right) v = 0$. This means that there exists $w\perp \one$ such that $\left(I - \one e_1^T\right)B_Gw = 0$, so $B_Gw \in \operatorname{null} \left(I - \one e_1^T\right) = \operatorname{span}(\one)$, thus there exists a real number $\gamma$ such that $B_G w = \gamma \one$. Conversely, if there exists a vector satisfying condition \ref{eq: 2}, then it is in the nullspace of $B_G'$ and is independent of $e_1$, so $\rank (B_G') \leq d$. 
\end{proof}
\begin{corollary} \label{nullspace}
    The nullspace of $B'_G$ consists of the vector $e_1$ and the space of vectors $w\in \onep$ such that $B_G(k)w = \gamma \one$ for some real number $\gamma$.
\end{corollary}
We now consider the conditions for a graph to be spherical. Denote by $J$ the all-ones matrix.

\begin{lemma} \label{3and4}
Let $G$ be a graph on $d+2$ vertices that is representable in $\rd$ with ratio $k$. It has a spherical representation with ratio $k$ if and only if the following two conditions hold:
\begin{enumerate}[label={$(\arabic*)$}]
    \setcounter{enumi}{2}
    \item $\det(xJ + B_G)$ is the zero polynomial.
    \item There exists $r_0 = r_0(B_G)\in \mathbb{R}$ such that $rJ+B_G  \succeq 0$ for all $r\geq r_0$.
\end{enumerate} 
\end{lemma}
\begin{proof}
    We first show the only if direction. Suppose that $G$ is a graph on $d+2$ vertices which can be representable in $\rd$ with ratio $k$, and its corresponding 2-distance set is spherical. Take the center $O$ of the sphere and take a point $O_1$ on the line in $\mathbb{R}^{d+1}$ through $O$ that is orthogonal to the subspace $\rd$ where $G$ is representable. The distance between every point from the 2-distance set and $O_1$ is the same and is denoted by $x_1$. Note that we can choose any $x_1\geq \sqrt{r_0/2}$ where $r_0$ is the radius of the circumsphere of the 2-distance set that corresponds to $G$. By Lemma~\ref{ns}, we see that $B' = 2x_1^2J+B_G \succeq 0$ for all $x_1 \geq \sqrt{r_0/2}$, so condition \ref{eq: 4} holds for $r\geq r_0$, and $\rank(B') \leq d+1$, i.e., $\det(B') = 0$. Since $\det(B')$ is a polynomial in $x_1$ which vanishes for all $x_1 \geq \sqrt{r_0/2}$, it must be the zero polynomial, so condition \ref{eq: 3} holds.
    
    We now show the if direction. Suppose that $G$ is representable in $\rd$ with ratio $k$ and conditions \ref{eq: 3} and \ref{eq: 4} hold. Take $r = r_0$. Since $rJ + B_G$ is positive semidefnite and $\rank(rJ+B_G) \leq d+1$, by Lemma~\ref{ns} the 2-distance set corresponding to $G$ is inscribed in a sphere with radius $r$ in $\mathbb{R}^{d+1}$ but it also lies on a subspace isomorphic to $\rd$, so it lies in the intersection of the sphere and an $d$-dimensional hyperplane, which is isomorphic to a sphere in $\rd$. This means that the 2-distance set corresponding to $G$ is spherical.
\end{proof}
By combining Proposition~\ref{1and2} and Lemma~\ref{3and4}, we obtain a set of four conditions that is necessary and sufficient for $G$ to be representable in $\rd$ with ratio $k$ using only the matrix $B_G$. 
\begin{enumerate}[label={(\arabic*)}]
    \item \label{eq: 1} For all vectors $w\in \one^{\perp}$ we have $w^TB_Gw \geq 0$.
    \item \label{eq: 2} There exists a nonzero vector $w\in \one^{\perp}$ such that $B_Gw = \gamma \one$ for a real number $\gamma$.
    \item \label{eq: 3} $\det(xJ + B_G)$ is the zero polynomial. 
    \item \label{eq: 4} There exists $r_0 = r_0(B_G)\in \mathbb{R}$ such that $rJ+B_G  \succeq 0$ for all $r\geq r_0$.
\end{enumerate}

We first simplify condition~\ref{eq: 3}:

\begin{lemma} \label{remr}
If a graph $G$ on $d+2$ vertices with a distance ratio $k$ satisfies condition \ref{eq: 2}, then \ref{eq: 3} holds if and only if $\det(B_G) = 0$.
\end{lemma}
\begin{proof}

It is clear that condition \ref{eq: 3} implies $\det(B_G) = 0$, so we only need to show the converse.

Assume that $\det(B_G) = 0$. Thus, there exists a vector $u \neq 0$ such that $B_Gu = 0$. From \ref{eq: 2}, there exists $w\in \one^{\perp}$ such that $B_Gw = \gamma \one$. It therefore suffices to show that for every $x$, there exists a nonzero vector $v$ such that $(xJ+B_G)v = 0$. 

First, if $\gamma = 0$, then $(xJ+B_G)w = 0$ for every $x$. Otherwise, take $\beta = -\frac{x \langle u, 1 \rangle}{\gamma}$. Then,
    \begin{align*}
        (xJ+B_G)(u+\beta w) & = xJu + \beta xJw + B_Gu + \beta B_Gw\\
        & = x \langle u, \one \rangle \one + \beta x \mathbf{0} + \mathbf{0} + \beta \gamma \one \\
        & = (x \langle u, \one \rangle + \beta \gamma)\one  = 0.
    \end{align*}
Note that $u+\beta w \neq 0$ because otherwise $u$ and $w$ would have been collinear, which gives us $Bw = 0$, so $\gamma = 0$. Thus, for every $x$ there exists a nonzero vector $v$ such that $(xJ+B_G)v = 0$, as needed.
\end{proof}
\section{Second eigenvalue of the adjacency matrix}\label{secondeigen}

From Lemma~\ref{remr}, we can find a formula for $B_G$ which does not depend on the ratio of the 2-distance set. Conversely, if this formula for $B_G$ holds, then we prove that condition \ref{eq: 1} implies condition \ref{eq: 2} and $\det(B_G) = 0$.  

First, we transform the matrix $B_G$ to a new matrix $\overline{B_G}$ which is more closely related to the adjacency matrix $A_G$ and rewrite conditions \ref{eq: 1}-\ref{eq: 4} in terms of $\bg$.

\begin{prop} \label{prop:newconds}
A graph $G$ on $d+2$ vertices is spherical in $\mathbb{R}^d$ with distance ratio $k$ if and only if the following conditions hold for $\bg = \bg(k) = \frac{1}{k^2-1} (B_G+J) = \frac{1}{k^2-1}I - A_G$:
\begin{enumerate}[label={$\overline{(\arabic*)}$}]
    \item \label{1bar} For all $w \in \onep$ we have $w^T\bg w \geq 0$.
    \item \label{2bar} There exists a nonzero vector $w\in \one^{\perp}$ such that $\bg w =\gamma \one$ for some real number $\gamma$.
    \item \label{3bar} $\det(\bg) = 0$.
    \item \label{4bar}There exists $r_0\in \mathbb{R}$ such that for all $r\geq r_0$ we have $rJ + \bg \succeq 0$.
\end{enumerate}
\end{prop}

\begin{proof}
It suffices to show that the conditions \ref{1bar}, \ref{2bar}, \ref{3bar}, and \ref{4bar} are equivalent to conditions \ref{eq: 1}, \ref{eq: 2}, \ref{eq: 3}, and \ref{eq: 4} respectively. 
The first two are clearly equivalent to the original conditions because $\bg w = \frac{1}{k^2-1}(B_G+J)w = \frac{1}{k^2-1} B_G w$ for all $w \in \onep$. 
Condition \ref{eq: 3} is equivalent to $\det(xJ +\bg) \equiv 0$ by a change of variable: \[\det \left(xJ + \frac{1}{k^2 - 1}(B_G + J) \right) = \left(\frac{1}{k^2-1}\right)^{d+2} \det \big((x(k^2-1)+1)J + B_G\big).\] By the same proof as in Lemma~\ref{remr} we see that this is also equivalent to condition \ref{3bar}. 
Condition \ref{4bar} is equivalent to $B_G + (r(k^2-1)+1)J \succeq 0$ for sufficiently large $r$, which is the same as condition \ref{eq: 4}.
\end{proof}

We now find a formula for $\bg$ in terms of the second eigenvalue of the adjacency matrix. This also gives us a formula for $B_G$ that does not depend on $k$.
\begin{lemma} \label{lbd2}
Let $G$ be a graph on $d+2$ vertices. Suppose that it satisfies conditions \ref{1bar},  \ref{2bar}, and \ref{3bar} for some $k>1$, then 
$B_G = I - \frac{1}{\lambda_2}A_G -J $.
\end{lemma}
\begin{proof}
Let $m = \frac{1}{k^2-1}$, so $\bg = mI - A_G$. Since $\det(\bg) = 0$, so there exists an eigenvector $v$ of $A_G$ that has eigenvalue $m = \lambda_i$ for some $i$. By condition \ref{1bar}, $x^T\bg x \geq 0$ for $x\in \onep$ which is equivalent to $x^TA_G x \leq mx^Tx$, i.e., $\frac{x^TA_Gx}{x^Tx} \leq m$ for every $x\in \onep$. So, by the Courant–Fischer–Weyl min-max principle \cite{horn2012matrix} we have
    $$\lambda_2 = \min_{\substack{U \\ \dim(U) = d-1}} \left\{ \max_{x\in U} \frac{x^TA_Gx}{x^Tx} \right\} \leq \max_{x\in \one^{\perp} } \frac{x^TA_Gx}{x^Tx} \leq m.$$
By condition \ref{2bar}, there exists a vector $x_1 \in \one^{\perp}$ such that $x_1^T(mI - A_G)x_1 = 0$ or equivalently $m = \frac{x_1^TA_Gx_1}{x_1^Tx_1}$. But $m = \lambda_i$ for some $i$, so $m\leq \lambda_1$.

If $\lambda_1 = \lambda_2$, then $m = \frac{x_1^TA_G x_1}{x_1^Tx_1}\leq \lambda_1 = \lambda_2$ by the min-max principle.

If $\lambda_1 > \lambda_2$, then $v_1 \not \in \onep$ because $\lambda_1$'s eigenvector $v_1$ has nonnegative coordinates. Since $\frac{x^TA_Gx}{x^Tx} = \lambda_1$ holds if and only if $x\in \operatorname{span}(v_1)$ and $x_1 \not \in \operatorname{span} (v_1)$, so $m < \lambda_1$ thus $m \leq \lambda_2$.

Therefore $\lambda_2 \leq m \leq \lambda_2$, so $m = \lambda_2$. Undoing the transformations of $B_G$, we obtain $B_G = I - \frac{1}{\lambda_2}A_G -J$, as needed.
\end{proof}

This gives us a formula for $k$ in terms of $\lambda_2$.

\begin{corollary}
Given graph $G$ on $d+2$ vertices. If $G$ is spherical, then the distance ratio is $\sqrt{\frac{1}{\lambda_2} + 1}$. (In particular, this implies that $\lambda_2 > 0$.)
\end{corollary}

\begin{proof}
From the proof of Lemma~\ref{lbd2}, we have $\frac{1}{k^2 - 1} = m = \lambda_2$, so we have $k = \sqrt{\frac{1}{\lambda_2} + 1}$.
\end{proof}

Therefore, we can rephrase Proposition \ref{prop:newconds} with the definition $\bg = \lambda_2 I - A_G$ instead of the original definition, which removes the need to consider $k$ altogether.

Using Lemma~\ref{lbd2} we can also prove that one of the eigenvectors of $\lambda_2$ is in $\onep$.

\begin{lemma}\label{v2}
Let $G$ be a graph on $d+2$ vertices with ratio $k$ that satisfies conditions \ref{1bar}, \ref{2bar}, and \ref{3bar}. 
There exists an eigenvector  $v_2$ of $\lambda_2$ that is orthogonal to $\one$.
\end{lemma}

\begin{proof}
Let $v_1$ be an eigenvector of $\lambda_1$ and $v_2$ be an eigenvector of $\lambda_2$ such that $\langle v_1, v_2 \rangle = 0$ and let $|| v_1|| = ||v_2|| = 1$. The subspace $\operatorname{span}(v_1, v_2)$ has dimension 2, so it intersects $\one^{\perp}$ in a line or a plane. This means that there exist real numbers $\alpha$ and $\beta$ such that $\alpha v_1 + \beta v_2 \in \one^{\perp}$ (where $\alpha$ and $\beta$ are not both equal to $0$). The matrix $\bg = \lambda_2 I - A_G$ by Lemma~\ref{lbd2}, so it has $v_2$ in its nullspace. By condition~\ref{1bar},
\begin{align*}
    0 & \leq (\alpha v_1 + \beta v_2)^T \bg (\alpha v_1 + \beta v_2)\\
    & = \alpha^2 v_1^T\bg v_1 + \beta^2 v_2^T\bg v_2+ \alpha \beta (v_1^T\bg v_2 + v_2^T\bg v_1) \\
    & = \alpha^2 (\lambda_2 v_1^T v_1 - v_1^TA_Gv_1) + \beta^2 v_2^T\bg v_2 + \alpha \beta v_1^T\bg v_2 + \alpha \beta (v_1^T\bg v_2)^T\\
    & = \alpha^2 (\lambda_2 v_1^T v_1 - v_1^TA_Gv_1)\\ 
    & = \alpha^2 (\lambda_2 - \lambda_1).
\end{align*}

Since $\alpha^2 (\lambda_2 - \lambda_1)$ is clearly nonpositive, this means that either $\lambda_1 = \lambda_2$ or $\alpha = 0$. In the first case $\lambda_2$'s eigenspace has dimension at least 2, so there is a nonzero intersection with $\one^{\perp}$, so we can set $v_2 \in \one^{\perp}$. In the second case we have $\beta v_2 \in \onep$, and $\beta \neq 0$, so $v_2 \in \onep$.
\end{proof}

\begin{corollary}\label{cor2}
    If $G$ is a graph with ratio $k$ that satisfies conditions \ref{1bar}, \ref{2bar}, and \ref{3bar} then if $\lambda_1 > \lambda_2$, all eigenvectors of $\lambda_2$ are in $\onep$, and if $\lambda_1 = \lambda_2$, then all but one vector of the eigenbasis of $\lambda_2$ can be chosen to be in $\onep$.
\end{corollary}

\begin{proof}
From the proof of Lemma~\ref{v2}, we can see that the eigenspace of $\lambda_2$ is in $\onep$ when $\lambda_1 > \lambda_2$. If $\lambda_1 = \lambda_2$, we can choose the eigenbasis of $\lambda_2$ so that at most one of the basis vectors is not in $\onep$, so we can assign it to be the eigenvector for $\lambda_1$.
\end{proof}

Now, we are ready to prove that condition \ref{1bar} implies conditions \ref{2bar} and \ref{3bar}, so we can remove the latter as the information from them is captured in the new formula for $\bg$.

\begin{lemma}\label{1l}
Given a graph $G$ on $d+2$ vertices and $\bg = \lambda_2 I - A_G$. If for every $w \in \onep$ $w^T\bg w \geq 0$, then conditions \ref{2bar} and \ref{3bar} are satisfied. 
\end{lemma}
\begin{proof}
    By the proof of Lemma~\ref{v2}, if condition \ref{1bar} is true for the matrix $\bg = \lambda_2 I - A_G$, then the eigenvector $v_2$ of the second largest eigenvalue is in $\onep$. From this we can deduce that $\bg v_2 = 0$, which implies both conditions \ref{2bar} and \ref{3bar}.
\end{proof}

\section{Reduction of the positive-semidefinite condition} \label{reduction}

In this section we simplify condition \ref{4bar}. To do this, we use condition~\ref{1bar} to reduce condition \ref{4bar} to the existence of an upper bound on $\displaystyle \frac{(\one^T A_G v)^2}{\lambda_2 v^T v - v^TA_G v}$ for $v\in \onep$. Lemma~\ref{pom} gives a necessary and sufficient condition for the existence of this bound.
\begin{lemma}\label{pom}
Let $G$ be a graph on $d+2$ vertices such that $\bg = \lambda_2 I - A_G$, and $\bg$ satisfies condition \ref{1bar}, and let $V$ be a subspace of $\onep$. There exists a real number $r$ such that
\begin{equation}\label{ineq}
    r(d+2)^2(\lambda_2v^Tv - v^TA_Gv) \geq (\one^T A_G v)^2 \text{ for every } v\in V
\end{equation}
if and only if $\one^TA_Gv = 0$ for all vectors $v\in V$ such that $\lambda_2 v^Tv = v^TA_Gv$.
\end{lemma}
\begin{proof}
    The forward direction is straightforward.
    For the reverse direction, we induct on the dimension of $V$. The base case $\dim (V) = 0$ is trivial.    For the inductive step, assume that $V$ has dimension at least one and that the statement holds for all subspaces of $\onep$ with lower dimension. Divide both sides of inequality \eqref{ineq} by $v^Tv$ to obtain 
    \begin{equation*}
        r(d+2)^2\left(\lambda_2 - \frac{v^TA_Gv}{v^Tv}\right) \geq \frac{(\one^T A_G v)^2}{v^Tv}.
    \end{equation*} By Cauchy-Schwartz’s inequality, we can bound the absolute value of the right-hand side $\left|\frac{(\one^TA_G v)^2}{v^Tv}\right| \leq ||\one^T A_G||^2$. Denote by 
    \begin{equation*}
    f(v) = \lambda_2 - \frac{v^TA_Gv}{v^Tv}.
    \end{equation*} 
    Let $L = \displaystyle \inf_{\substack{v\in V \\ v\neq 0}} f(v)$. Note that $L\geq 0$ by condition \ref{1bar}. If $L>0$, the inequality is true for every $r \geq \frac{||1^TA_G||^2}{(d+2)^2L}$. If $L = 0$, then as $f$ is invariant under scaling, so $L$ is also the infimum of $f$ on the unit sphere, which is compact. Therefore, this infimum is achieved, i.e., there exists a unit vector $w$ such that $f(w) = 0$ (so $\lambda_2 - w^TA_Gw = 0$). 
    
    We show that $A_G w = \lambda_2 w$, i.e.\ $w$ is an eigenvector of $A_G$. Note that \(P^T\overline{B}_G P\) is positive semidefinite, where \(P\) is the orthogonal projection matrix onto $\onep$, and $P^T = P$ holds. For $v\in \onep$, we have
    \begin{equation*}
        v^T \overline{B}_G v \;=\; v^T P^T \overline{B}_G P v \;=\; 0
    \end{equation*}
    if and only if $v$ is an eigenvector of $P^T\overline{B}_G P$ with eigenvalue $0$. Since $w$ is an eigenvector of $P^T\bg P$ with eigenvalue $0$, and hence
    \begin{equation*}
            0 = \bigl(P^T \bg P\bigr)w = P^T\bg w.
    \end{equation*}
    From the assumption $\bg w\in \onep$, it follows that
    \begin{equation*}
        \bg w = P^T\bigl(\bg w\bigr) = 0,
    \end{equation*}
    and hence $w$ is an eigenvector corresponding to $\lambda_2$, so $A_G w = \lambda_2 w$.

    By the inductive hypothesis, the inequality \eqref{ineq} holds for all $v \in V\cap \operatorname{span}(w)^{\perp}$. If $v$ is not perpendicular to $w$, by scaling $v$ we can write $v = w + \varepsilon x$ where $||x|| = 1$ and $x\perp w$. The inequality becomes 
    $$ r(d+2)^2(\lambda_2 + \lambda_2 \varepsilon^2 - w^TA_Gw - 2\varepsilon x^TA_Gw - \varepsilon^2 x^TA_Gx) \geq (\one^T A_G w + \varepsilon \one^T A_G x)^2 $$
    \begin{equation} \label{quadratic}
        r(d+2)^2(\lambda_2 \varepsilon^2 - 2\varepsilon x^TA_Gw - \varepsilon^2 x^TA_Gx) \geq (\varepsilon \one^T A_G x)^2.
    \end{equation} 

    As $x \perp w$, it follows that $x^T A_G w = \lambda_2 x^T w = 0$. Dividing both sides of the inequality \eqref{quadratic} by $\varepsilon^2$ gives \begin{equation*}
        r(d+2)^2(\lambda_2 - x^TA_Gx) \geq (\one^T A_G x)^2
    \end{equation*} for $x \in V\cap \operatorname{span}(w)^{\perp}$ that has lower dimension than $V$, so the inductive step is completed.
\end{proof}
We are now equipped to simplify condition \ref{4bar}. We prove that it is equivalent to condition \ref{1bar} combined with a new condition \ref{eq: 5}.
\begin{lemma}\label{ref4}
A graph $G$ on $d+2$ vertices is spherical if and only if $w^T\bg w \geq 0$ for all vectors $w \in \one^{\perp}$ where $\bg = \lambda_2 I - A_G$, and 
\begin{enumerate}[label={$\overline{(\arabic*)}$}]
\setcounter{enumi}{4}
\item \label{eq: 5} $\one ^T A_Gw = 0$ for all vectors $w\in \one^{\perp}$ such that $w^T\bg w = 0$.
\end{enumerate}
\end{lemma}
\begin{proof}
    We can easily see that condition \ref{4bar} implies condition \ref{1bar}, so if we know that  \ref{4bar} holds for the matrix $\bg = \lambda_2 I - A_G$, then by Lemma~\ref{1l} we also have conditions \ref{2bar} and \ref{3bar}, so $G$ has a spherical representation.
    
    Thus, we need to check that condition \ref{1bar} combined with the fact that  $\one^TA_Gw = 0$ for all vectors $w\in \one^{\perp}$ such that $\lambda_2 w^Tw = w^TA_Gw$ is equivalent to condition \ref{4bar}.
    
    Every vector $u$ can be expressed as $\alpha \one + \beta v$ where $v \in \onep$. So, condition \ref{4bar} states that for all $\alpha, \beta, v$ and sufficiently large $r$
    \begin{align*}
        0 & \leq (\alpha \one + \beta v)^T(rJ + \lambda_2 I - A_G)(\alpha \one + \beta v) \\
         & = r (\alpha \one + \beta v)^T J (\alpha \one + \beta v) + (\alpha \one + \beta v)^T(\lambda_2 I - A_G)(\alpha \one + \beta v).
    \end{align*}
    
    As $\operatorname{span}(J) = \operatorname{span}(\one)$, the right-hand side expression is equal to 
    \begin{align*}
    & \alpha^2 r(d+2)^2 + (\alpha \one + \beta v)^T(\lambda_2 I - A_G)(\alpha \one + \beta v) \\
    =\  & \alpha^2 r(d+2)^2 + \lambda_2 (\alpha \one + \beta v)^T(\alpha \one + \beta v) - (\alpha \one + \beta v)^TA_G(\alpha \one + \beta v)\\
    =\  & \alpha^2 r(d+2)^2 + \lambda_2 \alpha^2 (d+2) + \lambda_2 \beta^2v^Tv - (\alpha \one + \beta v)^TA_G(\alpha \one + \beta v)\\
    =\  & \alpha^2 r(d+2)^2 + \lambda_2 \alpha^2 (d+2) + \lambda_2 \beta^2v^Tv - \alpha^2 \one^TA_G\one - \beta^2 v^TA_Gv - 2\alpha \beta \one^T A_G v \\
    =\  & \alpha^2\left( \lambda_2  (d+2) - \one^TA_G\one + r(d+2)^2\right) - 2\alpha \beta \one^T A_G v + \beta^2(\lambda_2 v^Tv -  v^TA_Gv).
     \end{align*}
If $\alpha = 0$ this becomes condition \ref{1bar}. 

Assume that $\alpha\neq 0$. The first two terms in the coefficient of $\alpha^2$ are constants, so we can ignore them and increase $r$ by a constant. So, we want 
\begin{equation*}
    \alpha^2 r(d+2)^2 - 2\alpha \beta \one^T A_G v + \beta^2(\lambda_2 v^Tv -  v^TA_Gv) \geq 0.
\end{equation*}
Thus the discriminant of the quadratic polynomial should be non-positive. That is \begin{equation*}
    r(d+2)^2(\lambda_2 v^Tv -  v^TA_Gv) \geq (\one^T A_G v)^2 
\end{equation*}for sufficiently large $r$, which reduces the lemma to Lemma~\ref{pom}.
\end{proof}

We can simplify condition \ref{eq: 5} even further.

\begin{lemma}\label{last}
Let $G$ be a graph on $d+2$ vertices. Then $G$ has a spherical representation if and only if condition \ref{1bar} holds for $\bg = \lambda_2 I - A_G$, and for every vector $w \in \onep$ such that $w^T\bg w = 0$ we have $\bg w = 0$.
\end{lemma}
\begin{proof}

It suffices to show that if $w^T\bg w = 0$ for some $w\in \onep$ then $1^TA_Gw = 0$ if and only if $\bg w = 0$. Since $w^T\bg w = 0$ is equivalent to $w^TA_G w = \lambda_2 w^Tw$, by the proof of Lemma~\ref{pom} we see that if $A_Gw \in \onep$ then $A_Gw = \lambda_2 w$ and hence $\bg w = 0$. On the other hand, if $\bg w = 0$, then $\one^TA_Gw = \lambda_2 \one^T w = 0$.
\end{proof}

\section{Proof of the main theorem}\label{finsec}
We now use Lemma~\ref{last} to prove Theorem~\ref{mainres} and Proposition~\ref{lowestdimensionspherical}.

\begin{proof}[Proof of Theorem~\ref{mainres}]
    Let the eigenvalues of $PA_G P$ be $\mu_1\geq \mu_2 \geq \dots \geq \mu_{d+1}$ and 0 where 0 corresponds to the eigenvector $\one$. The matrices $PA_GP$ and $P\bg P$ have the same set of eigenvectors ($\bg = \lambda_2I - A_G$ and $P$ is a projection matrix). Suppose that the maximum eigenvalue of $PA_GP$ is $\lambda_2$. Therefore the spectrum of $P\bg P$ is the set $\lambda_2 - \mu_1, \lambda_2 - \mu_2, \dots, \lambda_2 - \mu_{d+1}, \lambda_2$. Thus all eigenvalues of $P\bg P$ are nonnegative, so condition \ref{1bar} holds for $\bg$. Conversely, if condition~\ref{1bar} holds, then $\mu_1 \leq \lambda_2$. Furthermore, by Lemma~\ref{1l}, conditions \ref{2bar} and \ref{3bar} hold. So, by Lemma~\ref{v2}, there is an eigenvector with an eigenvalue $\lambda_2$ of $A_G$ such that $v_2\in \onep$, thus $v_2$ is an eigenvector of $PA_GP$ and $P\bg P$, so $\mu_1 \geq \lambda_2$. Thus $\mu_1 = \lambda_2$.
    
    Now, by Lemma~\ref{last}, we have to prove that if we assume condition \ref{1bar} holds for $\bg$, then the following two conditions are equivalent.
    \begin{itemize}
        \item [(a)]The multiplicity of $\lambda_2$ in $A_G$ (excluding $\lambda_1$ if $\lambda_1 = \lambda_2$) is equal the multiplicity of $\lambda_2$ in $PA_G$.
        \item [(b)] We have $w^T\bg w = 0$ if and only if $\bg w = 0$ for all $w\in \onep$.
    \end{itemize}
    
    Condition \ref{1bar} holds for $\bg$, so by Lemma~\ref{1l}, conditions \ref{2bar} and \ref{3bar} hold, thus Corollary~\ref{cor2} holds and every eigenvector of $\lambda_2$ (except the eigenvector of $\lambda_1$ when $\lambda_1 = \lambda_2$) is also an eigenvector of $PA_GP$, thus the multiplicity of $\lambda_2$ in $A_G$ (excluding $\lambda_1$ when $\lambda_1 = \lambda_2$) is at most the multiplicity of $\lambda_2$ in $P A_G P$. By condition \ref{1bar} $w^T\bg w = 0$ if and only if $w$ is an null vector of $P\bg P$, i.e., an eigenvector of $\lambda_2$ in $A_G$. So, if $w^T\bg w = 0$ is true only if $\bg w = 0$, i.e., $w$ is an eigenvector of $\lambda_2$ in $A_G$, then the the multiplicity of $\lambda_2$ in $A_G$ (excluding $\lambda_1$ when $\lambda_1 = \lambda_2$) is at least the multiplicity of $\lambda_2$ in $P A_G P$, so they are equal. The reverse also holds.
\end{proof}
    We can combine \cref{mainres} with Corollary~\ref{nullspace} to obtain the lowest dimensional space in which a given 2-distance set is representable. 
\begin{proof}[Proof of Proposition~\ref{lowestdimensionspherical}]
    By Lemma~\ref{main}, $G$ is representable in $\rn$ if and only if $\rank (B'_G) \leq n$ where $B'_G= \left(I - \one e_1^T\right)B_G\left(I - e_1\one^T\right).$ Let $S$ be the space of vectors $w\in \onep$ such that $B_Gw = \gamma \one$ for some real number $\gamma$. By Corollary~\ref{nullspace}, $\rank (B'_G) = d+2 - 1 - \rank(S) = d+1 - \rank(S)$. Note that the nullspace of $PB_GP$ which is $\one \cup S$, thus $\rank(PB_GP) = d+2 - 1 -\rank (S) = d+1-\rank(S) = \rank (B'_G)$, so $\rank (B'_G) \leq n$, but the nullspace of $PB_GP$ is the union of $\one$ and eigenspace of $\lambda_2$ in $PA_GP$.  For spherically representable graphs the multiplicity of $\lambda_2$ in $PA_GP$ is the same as the multiplicity of $\lambda_2$ in $A_G$. Therefore, $G$ is representable in $\rn$ if the multiplicity of $\lambda_2$ in $A_G$ is at least $d+1-n$.
\end{proof}
We also give an equivalent formulation of \cref{mainres} using the Cauchy interlacing theorem, which gives a more intuitive understanding of what \cref{mainres} means.
\begin{theorem}[Cauchy interlacing theorem \cite{hwang2004cauchy}]
    Let $A$ be a symmetric $n \times n$ matrix and $B$ be an $m \times m$ matrix with $B = PAP^*$ where $P$ is an orthogonal projection onto a subspace of dimension $m$. Then if the eigenvalues of $A$ are $\alpha_1 \geq \alpha_2 \geq \dots \geq \alpha_n$ and the eigenvalues of $B$ are $\beta_1 \geq \beta_2 \geq \dots \geq \beta_m$, then for all $j \leq m$
    $$\alpha_j \geq \beta_j \geq \alpha_{j+n-m}.$$
\end{theorem}
In our case, if the eigenvalues of $PA_GP$ are $\mu_1 \geq \mu_2 \geq \dots \geq \mu_{n-1}$, then \[\lambda_1 \geq \mu_1 \geq \lambda_2 \geq \mu_2 \geq \dots \geq \lambda_k \geq \mu_k \geq \lambda_{k+1} \geq \dots \geq \lambda_n.\]
The condition of \cref{mainres} is therefore equivalent to $\mu_1 = \lambda_2$ and $\mu_k < \lambda_k$, where $k$ is such that $\lambda_{k+1} < \lambda_k = \lambda_2$. Proposition~\ref{lowestdimensionspherical} implies that the lowest representable dimension of a spherical graph on $n$ vertices is $n - k$.

The main theorem also implies that all small regular graphs have a spherical representation: 
\begin{corollary}
    If $G$ is a regular graph on $d+2$ vertices that is not complete multipartite, then $G$ has a spherical representation in $\rd$. 
\end{corollary}
\begin{proof}
    If $G$ is a regular graph, then the eigenvector of the first eigenvalue is $\one$, thus $\mu_i = \lambda_{i+1}$ for $i = 1,2\dots, n-1$, which clearly satisfies \cref{mainres}.
\end{proof}

It may be possible to derive more necessary or sufficient combinatorial conditions for a graph to be spherical based on our main theorem.

\smallskip

\paragraph{\bfseries Acknowledgments.} 
The bulk of this research was conducted during Research Science Institute (RSI) hosted by the Center for Excellence in Education (CEE) at Massachusetts Institute of Technology (MIT), where the first author was a participant and the second author was a research mentor. We thank Dr.~Tanya Khovanova, Prof.~David Jerison and Prof.~Ankur Moitra for helpful discussions during the RSI program. We thank Eli Meyers for his help with various technical issues during the project, and Dr.~John Rickert and Yunseo Choi for their feedback on earlier versions of this paper. We are also grateful to the anonymous reviewer for suggesting a simpler proof that \(w\) is an eigenvector of \(A_G\). We are thankful to HSSIMI, Sts.\ Cyril \& Methodius International Foundation, Foundation for Support of Children Mariyana Palanchova, and the MIT Math Department for funding our research.